\documentclass{article}[12pt]
\usepackage{color}
\usepackage{graphicx}
\usepackage{amssymb,amsmath,amsthm,amsfonts}
\usepackage{verbatim}

\newtheorem{theorem}{Theorem}[section]

\newtheorem{lemma}[theorem]{Lemma}

\theoremstyle{definition}

\title{Energy growth for a nonlinear oscillator coupled to a monochromatic wave}
\date{\today}
\author{Dmitry Turaev, Christopher Warner, Sergey Zelik\\ \small{Imperial College, London, and University of Surrey}}

\begin{document}
\maketitle

\subsubsection*{Abstract}
{\small A system consisting of a chaotic (billiard-like) oscillator coupled to
a linear wave equation in the three-dimensional space is considered. It is shown that the chaotic behaviour of the oscillator
can cause the transfer of energy from a monochromatic wave to the oscillator, whose energy can grow without bound.}

\section{Setting the problem and result}
The system we consider (a linear wave coupled to an oscillator) is formally defined by the Hamiltonian
\begin{equation}
H=\frac{1}{2}\left(p_y^2+p_z^2 \right) + V(y,z) +\epsilon k(y,z)\int_{\|{\bf x}\|\leq1} u({\bf x}, t)\ \text{d}^3{\bf x}
+ \frac{1}{2}\int\left(u_t^2 + (\nabla_x u)^2 \right)\ \text{d}^3{\bf x}, \label{FullEnergyFunctional}
\end{equation}
where $u({\bf x},t)$, the massless Klein-Gordon field, is a scalar function on $\mathbb{R}^3\times \mathbb{R}^1$, and $(y,z)\in\mathbb{R}^2$ are coordinates
in the configuration space of the oscillator ($(p_y=\dot y,p_z=\dot z)\in\mathbb{R}^2$ are the corresponding momenta).
The smooth potential $V(y,z)$ is bounded from below and tends to infinity as $\|y,z\|\to\infty$. To avoid technicalities, we
assume that $V$ equals to infinity outside a bounded domain in the $(y,z)$ plane.
The interaction coefficient $k(y,z)$ is smooth and bounded along with the first derivatives, and $\epsilon$ is small.

The corresponding equations of motion are
\begin{equation}
u_{tt}-\Delta u=-\epsilon k(y(t),z(t))\xi({\bf x}), \label{WaveEquation}
\end{equation}
where $\xi({\bf x})=\left\{ \begin{array}{cc} 1 & \|{\bf x}\|\leq1 \\
0 & \|{\bf x}\|>1 \end{array}\right.$ (the characteristic function of the unit ball in the $\bf x$-space), and
\begin{equation}\begin{split}
\ddot y + \frac{\partial}{\partial y}\left(V(y,z) \right) &= -\epsilon k'_y \int_{\|{\bf x}\|\leq1}u({\bf x},t)\ \text{d}^3{\bf x}, \\
\ddot z + \frac{\partial}{\partial z}\left(V(y,z) \right) &= -\epsilon k'_z \int_{\|{\bf x}\|\leq1}u({\bf x},t)\ \text{d}^3{\bf x}. \label{System1}
\end{split}\end{equation}
We may think of the oscillator as being located in the unit ball in the $\bf x$-space and emitting/receiving the wave $u({\bf x}, t)$.
It can be shown (similar to \cite{label41,ScatteringTheory,treschev}) that if the initial energy of the wave is finite
(i.e. $\int\left(u_t^2 + (\nabla_x u)^2 \right)\ \text{d}^3{\bf x}<\infty$ at $t=0$), then
the field $u$ tends to a constant and, for a typical choice of the interaction coefficient $k(y,z)$,\footnote{For a general choice of $k(y,z)$
the state of oscillator may approach an invariant set on which the function $k(y,z)$ stays constant, see more in \cite{label41,treschev}} the energy of the oscillator
decreases, and it comes to a rest at some stationary point of the potential function as $t\to+\infty$, i.e. the energy flows from the oscillator
to the field and is carried away to infinity. In this paper we show that if the wave has infinite energy, then an opposite process may take place.
Namely, when the oscillator operates in a chaotic regime, the energy can be pumped from the field to the oscillator, and the oscillator's
kinetic energy can increase up to any given value.

The field equation (\ref{WaveEquation}) can be explicitly resolved:
\begin{equation}
u({\bf x},t)= u_0({\bf x},t)+ \frac{\epsilon}{4\pi}\int \frac{k(z(t-\|{\bf s}\|),y(t-\|{\bf s}\|))}{\|s\|}\xi({\bf x-s})\ \text{d}^3{\bf s}, \label{WaveSolution}
\end{equation}
where $u_0$ is the solution to the homogeneous wave equation. We take as $u_0$ a monochromatic standing wave of infinite energy:
\begin{equation}\label{sphst}
u_0({\bf x},t)=K\sin(\omega t) \int \alpha_{\bf k} \cos ({\bf k} \cdot {\bf x})\; \delta (\omega-\|{\bf k}\|) d^3{\bf k}
\end{equation}
(the fact that it is a standing wave is not very important; the fact that the frequency spectrum is discrete and finite is used in an essential way).
By placing (\ref{WaveSolution}) into (\ref{System1}) we obtain a system of delayed differential equations
\begin{equation}\begin{split}
\ddot y + \frac{\partial}{\partial y}\left(V(y,z) \right) = &-A_\omega \epsilon k'_y(y,z) \sin(\omega t)-\epsilon^2 k'_y\int_0^2 k(z(t-s),y(t-s))\mathcal{P}(s)\ \text{d}s, \\
\ddot z + \frac{\partial}{\partial z}\left(V(y,z) \right) = &-A_\omega \epsilon k'_z(y,z) \sin(\omega t)-\epsilon^2 k'_z\int_0^2 k(z(t-s),y(t-s))\mathcal{P}(s)\ \text{d}s,
\label{System2}\end{split}\end{equation}
where
\begin{equation}\label{awww}
A_\omega=K \int \alpha_{\bf k} \xi({\bf x}) \cos ({\bf k} \cdot {\bf x}) \;\delta (\omega-\|{\bf k}\|) d^3{\bf k} d^3 {\bf x},
\end{equation}
and $\displaystyle \mathcal{P}(s)=\frac{\pi}{12} (16-12s+s^3)s$ (we denote here $s=\|\bf s\|$).
This is a Hamiltonian system, subject to a conservative periodic perturbation of order $\epsilon$ and a dissipative\footnote{We are
not showing here that the delay term always leads to a dissipation of energy. In order to get some insight, one can check that in the linear case
(i.e. quadratic potential $V$ and linear coefficient of interaction $k$) the addition of this term shifts the spectrum to the left of the imaginary axis,
i.e. creates dissipation (however, as we consider bounded functions $k$ and the potential $V$ that is infinite outside a bounded domain, the linear case
is not the subject of this paper).} correction of order $\epsilon^2$. Namely, if we drop
the $O(\epsilon^2)$-terms (those including the delay), system (\ref{System2}) will be defined by a time-dependent Hamiltonian function
\begin{equation}\label{hamnd}
H=\frac{1}{2}\left(p_y^2+p_z^2 \right) + V(y,z) +\epsilon A_\omega k(y,z) \sin(\omega t).
\end{equation}
In general, the systems with a time-dependent Hamiltonian do not preserve energy, and the periodic forcing may, in fact, lead to an unbounded growth of energy.
Importantly, as the kinetic energy grows, the forcing becomes effectively slow. It is known \cite{HamiltonianGrowth,D1,D2,label1} that if a Hamiltonian system behaves
chaotically at all sufficiently large energies, then adding a slow periodic forcing creates orbits of unbounded energy growth. Similar to \cite{label1},
we will show that in the system of type (\ref{hamnd}) this growth is linear in time; namely, one may construct orbits for which
the energy gain per period of the force is bounded from below by a non-zero constant of order $\epsilon$. Using the fact that the delay terms which can lead to a dissipation
of energy are of order $\varepsilon^2$, i.e. they are much smaller than the energy gain, we show below that the full system (\ref{System2}) has, for sufficiently small
$\epsilon$, solutions for which the energy grows up to any given value, linearly in time.

Our method does not allow us to show the existence of orbits for which the energy tends to infinity (the larger the energy value we want to achieve, the smaller
value of $\epsilon$ we have to take). Similarly, our approach is not applicable to the case of non-zero mass in the wave equation. However, we conjecture
that such orbits do exist in system (\ref{System2}),
like they do in the time-dependent Hamiltonian system defined by (\ref{hamnd}) (see \cite{label1}). We also think that our approach (based
on a reduction to an invariant manifold) is interesting in its own right. We stress that the mechanism of the sufficiently fast (linear) energy growth we obtain here
is based on the chaotic behaviour of the oscillator. Would the oscillator be integrable, the existence of adiabatic invariants \cite{AKN,Ar} would
impede a possible energy growth in the shortened system (\ref{hamnd}), and it seems plausible that the $O(\epsilon^2)$-dissipation due to delayed terms should
arrest the energy growth completely in this case. So, the picture we propose is the following: while an integrable oscillator that interacts with a monochromatic wave
(in $\mathbb{R}^3$) should lose energy, a chaotic oscillator may take the energy from the wave. Note that the reported effect cannot be immediately interpreted as a resonant
phenomenon: since the potential is infinite outside a bounded domain, the region of allowed motions (the Hill's region) in the configuration plane $(y,z)$
is always bounded, and since the velocities $p_y$ and $p_z$ tend to infinity as the energy grows, it follows that the characteristic return times tend to zero
and become much smaller than the period of external force.

In order to formulate the result precisely, we need to define what do we mean by the "chaotic oscillator". Consider a system of ODE's defined by the Hamiltonian
\begin{equation}\label{hamn0}
H=\frac{1}{2}\left(p_y^2+p_z^2 \right) + V(y,z)
\end{equation}
(the Hamitlonian (\ref{hamnd}) at $\epsilon=0$). As we mentioned, we assume that $V$ is infinite outside a certain bounded region $D$. Following \cite{label1},
assume that for each
sufficiently large $h$, the system defined by (\ref{hamn0}) has, in the energy level $H=h$, a pair of hyperbolic periodic orbits $L_a$ and $L_b$ such that
the unstable manifold of $L_a$ has an orbit $\Gamma_{ab}$ of transverse intersection with the stable manifold of $L_b$, and
the unstable manifold of $L_b$ has an orbit $\Gamma_{ba}$ of transverse intersection with the stable manifold of $L_a$. Moreover, we assume that the orbits
$L_a$, $L_b$, $\Gamma_{ab}$ and $\Gamma_{ba}$ depend continuously on $h$. When these conditions are fulfilled we call system (\ref{hamn0}) a chaotic oscillator.
It is well known \cite{Shilnikov, AShi} that the existence of the transverse heteroclinic cycle implies a chaotic behaviour indeed. Namely, fix a value of $h$, and
take a sufficiently small neighbourhood $U$ of the heteroclinic cycle $L_a\cup L_b \cup \Gamma_{ab} \cup \Gamma_{ba}$ in the energy level $H=h$.
Then the set of all orbits that stay in $U\cap \{H=h\}$ is in one-to-one correspondence with arbitrary sequences of the symbols $a$ and $b$ (one round made by the orbit
near $L_a$ is coded by $a$, and a round near $L_b$ is coded by $b$).

The basic example of chaotic oscillators in our setting is given by systems with billiard-like potentials.
Namely, let $D\subset\mathbb{R}^2$ be the bounded domain
such that $V(y,z)$ is finite inside $D$ and infinite outside of $D$. Let the boundary of $D$ consist of a finite number of smooth arcs, $S_1,\dots,S_n$,
joined at corner points. Let $C^r$-smooth functions
$Q_1(x,y),\dots,Q_n(x,y)$ be such that, for each $j=1,\dots,n$, the function $Q_j$ is defined in a neighbourhood
of the boundary arc $S_j$, the arc $S_j$ is a level line of $Q_j$ (i.e.  $Q_j(x,y)|_{(x,y)\in S_j}=const$), and $\nabla Q_j\neq 0$ in the neighborhood of $S_j$.
We will call the potential $V$ that equals to infinity outside of $D$ a billiard-like potential if there exist
$C^r$-smooth, strictly monotonic functions $W_1(Q),\dots, W_n(Q)$ such that for each $j=1,\dots,n$ the potential $V$ in a small neighbourhood of the arc $S_j$ is given by
\begin{equation}\label{pot}
V(x,y)=W_j(Q_j(x,y))
\end{equation}
(we do not include the corner points into the arcs $S_j$, i.e. they are open intervals, so their small open neighbourhoods do not need to contain the corner points,
hence even if two arcs join at a corner point, their small neighbourhoods where (\ref{pot}) holds do not need to intersect, i.e. no relation between the corresponding
functions $W_j$ arises). By scaling the momenta $p_{y,z}$ to $\sqrt{h}$, system (\ref{hamn0}) on the energy level $H=h$ transforms into the system
\begin{equation}\label{hamnh}
H=\frac{1}{2}\left(p_y^2+p_z^2 \right) + \frac{1}{h}V(y,z)
\end{equation}
on the energy level $H=1$. One can easily check that if $V$ is a billiard-like potential (i.e. it is defined by (\ref{pot}) near the boundary arcs), then
the family of Hamiltonians (\ref{hamnh}) (with $h^{-1}$ being a small parameter) satisfies conditions of \cite{label4} which guarantee that the flow defined
by such Hamiltonian is an approximation, at $h$ large enough, to the billiard flow in $D$. Namely, billiard in $D$ is the mechanical dynamical system
which describes the following motion of a point mass \cite{label10}: the particle moves inertially (with a constant velocity) inside the region $D$ in a plane until it hits
the boundary of $D$, then the particle is reflected according to the elastic reflection law, the angle of reflection is the angle of incidence, and so on.
As it follows from \cite{label4}, the time-shift maps by the billiard flow in $D$ and by the Hamiltonian flow defined by (\ref{hamnh}) with a billiard-like potential $V$
and $h$ large enough are $C^r$-close outside the
set of singular billiard orbits (i.e. such billiards orbits which enter a corner point or are tangent to the boundary of the billiard domain $D$). Therefore, if the billiard
in $D$ has a heteroclinic cycle of two hyperbolic periodic orbits $L_a$ and $L_b$ and two transverse heteroclinic orbits $\Gamma_{ab}$ and $\Gamma_{ba}$ and neither of these
orbits is singular, then system (\ref{hamnh}) in the energy level $H=1$ (so system (\ref{hamn0}) in the energy level $H=h$) also has such heteroclinic cycle for all $h$ large
enough, i.e. system (\ref{hamn0}) with a billiard-like potential $V$ is a chaotic oscillator according to our definition.
In particular, the oscillator defined by a billiard-like potential is chaotic when the underlying billiard is dispersive: all the boundary arcs are concave and meet each other
at non-zero angles at the corner points \cite{Sinai,BS,label11,label10,label4}.

Let $L_a$ and $L_b$ be the two non-singular hyperbolic periodic orbits of the billiard in $D$ that are connected by the transverse non-singular heteroclinics $\Gamma_{ab}$
and $\Gamma_{ba}$. Fix the speed of the particle in the billiard to be equal to $1$, and let $T_c$ (where $c=a,b$) be the period of $L_c$, and
$(y_c(t),z_c(t))|_{t\in[0,T_c]}$ be the equation of the orbit $L_c$. Denote
\begin{equation}\label{spdvl}
v_c=\frac{1}{T_c}\int_0^{T_c}k(y_c(t),z_c(t)) dt.
\end{equation}
Assume
\begin{equation}\label{vabc}
v_a\neq v_b.
\end{equation}

\begin{theorem}\label{mainthm}
Consider a system (\ref{WaveEquation}), (\ref{System1}) that describes a linear massless scalar field interacting with a chaotic oscillator,
with a billiard-like potential $V$ and the interaction coefficient $k(y,z)$ bounded with first derivatives.
Let condition (\ref{vabc}) hold. Then there exists $h_0$ such that for each $\omega$ for which $A_\omega\neq 0$ (see (\ref{awww}))
and for any $h_1>h_0$ there exists $\epsilon_0(h_1)>0$ such that for all $\epsilon\in (0,\epsilon_0)$ the system has a solution
with the wave component $u$ given by (\ref{WaveSolution}),(\ref{sphst}) and the oscillator component $(y(t),z(t),p_y(t),p_z(t))$ reaching from
the energy $h_0$ at $t=0$ to the energy $h_1$ at some finite $t$ (the energy of the oscillator is given by (\ref{hamn0})).
\end{theorem}

The proof of the theorem occupies the rest of the paper. In Section \ref{invm} we prove an invariant manifold theorem which allows us to
reduce the system of delayed differential equations (\ref{System2}) to a four-dimensional non-autonomous system of ODE's, which is $O(\epsilon^2)$-close
to the Hamiltonian system given by (\ref{hamnd}). This reduction is possible on any bounded set of values of $(y,z,p_y,p_z)$ for sufficiently small $\epsilon$.
Note that the maximal value of $\epsilon$ for which we can guarantee the reduction increases as $p_y$ and $p_z$ grow, which is one of the reasons
why we show only bounded energy growth in Theorem \ref{mainthm}. In Section \ref{energy} we apply the construction of \cite{label1} (modified for the non-Hamiltonian
case) to the reduced system, and finish the proof.

\section{Invariant Manifold Theorem}
\label{invm}
In this section we reduce the infinite-dimensional system of differential equations with delay to a finite-dimensional system of ordinary
differential equations. Consider the system
\begin{equation}
\dot X(t)=F(X(t)) - \delta\int_0^\tau G(X(t-s),s)\ \text{d}s, \label{DelayEquation}
\end{equation}
where $X$ belongs to an $n$-dimensional smooth manifold $\mathbb{M}$, and the functions $F$ and $G$ are $C^r$-smooth. Let $F$ be such that the differential equation
\begin{equation}\label{dl0}
\dot X=F(X)
\end{equation}
has, for any initial condition, a solution defined for all $t\in(-\infty,+\infty)$ (for example, this differential equation is Hamiltonian with compact energy levels).
System (\ref{System2}) can be represented in this form with $\delta=\epsilon^2$. Namely, we introduce a variable $\theta\in \mathbb{S}^1$ and
rewrite (\ref{System2}) as
\begin{equation}
\begin{array}{l} \displaystyle
\dot y=p_y, \qquad \dot z=p_z, \qquad \dot\theta=\omega, \\ \displaystyle
\dot p_y = -\frac{\partial}{\partial y}\left(V(y,z) \right) - A_\omega \epsilon k'_y(y,z) \sin(\theta)-\epsilon^2 k'_y\int_0^2 k(z(t-s),y(t-s))\mathcal{P}(s)\ \text{d}s, \\ \\
\displaystyle
\dot p_z = - \frac{\partial}{\partial z}\left(V(y,z) \right) -A_\omega \epsilon k'_z(y,z) \sin(\theta)-\epsilon^2 k'_z\int_0^2 k(z(t-s),y(t-s))\mathcal{P}(s)\ \text{d}s,
\label{System3}\end{array}\end{equation}
so $X=(y,z,p_y,p_z,\theta=\omega t)\in \mathbb{R}^4\times \mathbb{S}^1$ here.

It is well known that given any continuous function $\hat X:[-\tau,0]\rightarrow R^n$ there exists a unique solution $X(t)$ of equation (\ref{DelayEquation})
such that $X(t)\equiv \hat X(t)$ at $t\in[-\tau,0]$. One can therefore view the evolution defined by equation
(\ref{DelayEquation}) as a semiflow in the space $C$ of continuous functions that act from $[-\tau,0]$ to $R^n$: the time-$s$ map of the semiflow takes
the initial condition $X(s)|_{[-\tau,0]}$ to the segment $X(s)|_{[t-\tau,t]}$ of the corresponding solution. A smooth function
$\mu: R^n\times [0,\tau]\to R^n$ defines a map $R^n\rightarrow C$ by the rule $X(s)=\mu(x,-s)$; the graph of such map is an invariant manifold for the semiflow
defined by (\ref{DelayEquation}) if the solution with the initial condition $\mu(X(0),-s)|_{s\in[-\tau,0]}$ satisfies
\begin{equation}\label{invc}
\mu(X(t),s)=X(t-s)
\end{equation}
for all $t\geq 0$ and $s\in[0,\tau]$. When such manifold exists, the restriction of system (\ref{DelayEquation}) onto it is a system of ordinary differential equations
\begin{equation}
\dot X= F(X) - \delta\int_0^\tau G(\mu(X,s),s)\ \text{d}s. \label{lemma1}
\end{equation}
In other words, the existence of a smooth function $\mu$ which satisfies the invariance condition (\ref{invc}) implies the existence of an $n$-parameter family of solutions
(parameterised by $X(0)$) to the delayed differential equation (\ref{DelayEquation}) which also solve the ordinary differential equation (\ref{lemma1}).

\begin{lemma}\label{TheLemma}
Given any compact subset $K$ of the $X$-space $\mathbb{M}$, for all sufficiently small $\delta$ there exists a $C^{r-1}$-smooth function $\mu: K\times [0,\tau]\to R^n$ such that
for any $X(0)\in K$ the solution $X(t)$ of (\ref{DelayEquation}) which starts with the initial condition $\mu(X(0),-s)|_{s\in[-\tau,0]}$ satisfies the invariance condition
(\ref{invc}) for the interval of $t$ values for which the solution stays in $K$.
\end{lemma}

\begin{proof}
By the Cauchy-Picard-Lindel\"of theorem \cite{difur}, given any smooth functions $\mu$, $F$ and $G$, the ordinary differential equation (\ref{lemma1}) generates
a uniquely defined (on any given finite time interval) solution
for any initial condition from $K$ if $\delta$ is small enough. This solution depends smoothly on any parameter on which the system
depends smoothly. In particular, the solution depends smoothly on the function $\mu$. Thus, we will show below that the solution is
a $C^1$-function of $\mu$ which is considered as an element of the space of $C^{r-1}$-smooth functions.

Given $\mu$ and $\delta$, take any $X(t)\in K$, consider its backward orbit by equation (\ref{lemma1}), and let $X(t-s)$ be the point on this orbit
which corresponds to the (backward) shift to time $s$. Denote as $\phi(\mu,\delta)$ the map $K\times [0,\tau]\to R^n$ which (for given $\mu$ and $\delta$)
sends $X(t)\in K$ and $s\in[0,\tau]$ to $X(t-s)$, i.e. $\phi(\mu,\delta)$ is the backward flow of the ordinary differential equation (\ref{lemma1}).
The flow has the same smoothness as the equation, so if $\mu\in C^{r-1}$, then $\phi(\mu,\delta)$ is $C^{r-1}$ with respect to $X$ and $s$.
At $\delta=0$ the flow $\phi(\mu,\delta)$ is generated by equation (\ref{dl0}), so it is independent of $\mu$. We will show in a moment that $\phi$ depends $C^1$-smoothly
on $\mu$ and $\delta$. As $\phi(\mu,0)$ does not depend on $\mu$, it follows that the Frechet derivative $\displaystyle\frac{\partial \phi}{\partial \mu}$ vanishes at $\delta=0$.
Hence, by the Implicit Function Theorem, the equation
\begin{equation}\label{mphd}
\mu=\phi(\mu,\delta),
\end{equation}
has, for every small $\delta$, a unique solution $\mu\in C^{r-1}$. If we plug this particular $\mu$ into the right-hand side of (\ref{lemma1}), then
equation (\ref{mphd}) will exactly mean that condition (\ref{invc}) is satisfied by the solutions of (\ref{lemma1}). Thus, each of these solutions will
also solve the original delayed equation (\ref{DelayEquation}), which is the statement of the lemma.

Thus, to finish the proof, it remains to show the smooth dependence of the flow of (\ref{lemma1}) on $\mu$. It is well known that the solutions
of ordinary differential equations depend smoothly on the function in the right-hand side of the equation. So we are left to show that the so-called
Nemytsky operator (or substitution  operator) $\mathcal{N}$ which takes the function $\mu(X,s)$ to the function $G(\mu(X,s),s)$ is of class $C^1$
on the space of $C^{r-1}$-functions $\mu$, provided $G$ is $C^r$ as a function of $(X,s)$.
In order to do this, it is enough to check that the Frechet derivative of $\mathcal{N}$ at
a given function $\mu$ is the operator of multiplication to $\displaystyle\frac{\partial G}{\partial X}(\mu(X,s),s)$. For that, one needs to check that
$$\|G(\mu+\Delta\mu)- G(\mu)-G'(\mu)\cdot\Delta\mu\|_{_{C^{r-1}}}=o(\|\Delta\mu\|_{_{C^{r-1}}})$$
(we suppress, notationally, the dependence of $s$, so $G'$ denotes here the derivative of $G$ with respect to its first argument).
This relation is rewritten as
$$\left\|\int_0^1 (G'(\mu+\xi\Delta\mu)-G'(\mu))\ \text{d}\xi \cdot\Delta\mu \right\|_{C^{r-1}} =o(\|\Delta\mu\|_{_{C^{r-1}}})$$
which reduces to the obvious (since $G$ is $C^r$) claim that
$G'(\mu+\xi\Delta\mu)_{\stackrel{\longrightarrow}{C^{r-1}}} G'(\mu)$ as
$\Delta\mu_{\stackrel{\longrightarrow}{C^{r-1}}} 0$, uniformly for all $\xi\in[0,1]$.
\end{proof}

Applying this lemma to system (\ref{System3}), we find that for all sufficiently small $\epsilon$
it has a family of solutions which satisfy the system of ordinary differential equations
\begin{equation}\begin{split}
\ddot y + \frac{\partial}{\partial y}\left(V(y,z)\right) = &-A_\omega \epsilon k'_y(y,z) \sin(\theta)+\epsilon^2 \mathcal{F}_1(y,z,\dot y,\dot z,\theta,\epsilon),\\
\ddot z + \frac{\partial}{\partial z}\left(V(y,z)\right) = &-A_\omega \epsilon k'_z(y,z) \sin(\theta)+\epsilon^2 \mathcal{F}_2(y,z,\dot y,\dot z,\theta,\epsilon),\\
\dot \theta = & ~\omega,
\label{System4}\end{split}\end{equation}
where the smooth functions $\mathcal{F}_{1,2}$ incorporate the delay terms. Moreover, for all sufficiently small $\epsilon$,
every solution of this system satisfies the original  delayed equations (\ref{System2}) on the time interval the solution stays in a bounded ball in the
phase space (we may take this ball as large as we want; however, to increase the radius of the ball, we might need to take $\epsilon$ smaller).

\section{A partially-hyperbolic set and the energy drift}
\label{energy}

As we just have shown, it is enough to establish the existence of the solutions of growing energy in system (\ref{System4}).  By taking a sufficiently large $h_0$
and scaling time to $\sqrt{h_0}$, we rewrite this system as

\begin{equation}\begin{split}
\ddot y + \frac{1}{h_0}\frac{\partial}{\partial y}\left(V(y,z)\right) = &-A_\omega \frac{\epsilon}{h_0} k'_y(y,z) \sin(\theta)+
\frac{\epsilon^2}{h_0} \mathcal{F}_1(y,z,\sqrt{h_0}\dot y,\sqrt{h_0}\dot z,\theta,\epsilon),\\
\ddot z + \frac{1}{h_0}\frac{\partial}{\partial z}\left(V(y,z)\right) = &-A_\omega \frac{\epsilon}{h_0} k'_z(y,z) \sin(\theta)+
\frac{\epsilon^2}{h_0} \mathcal{F}_2(y,z,\sqrt{h_0}\dot y,\sqrt{h_0}\dot z,\theta,\epsilon),\\
\dot \theta = & ~\frac{\omega}{\sqrt{h_0}},
\label{System5}\end{split}\end{equation}

At $\epsilon=0$ the first two equations are independent on $\theta$:
\begin{equation}\begin{split}
\ddot y + \frac{1}{h_0}\frac{\partial}{\partial y}\left(V(y,z)\right) = & 0,\\
\ddot z + \frac{1}{h_0}\frac{\partial}{\partial z}\left(V(y,z)\right) = & 0.
\label{System6}\end{split}\end{equation}
As we mentioned, the assumption that our billiard-like oscillator is chaotic means that this system has, at sufficiently large $h_0$,
a uniformly-hyperbolic set $\Lambda$ in every energy level $H\geq 1$ (where the rescaled energy $H$ is given by (\ref{hamnh}) with $h=h_0$).
Namely, we take the two non-singular hyperbolic periodic orbits  $L_a$ and $L_b$  of the billiard in $D$ and the two
transverse non-singular heteroclinics $\Gamma_{ab}$ and $\Gamma_{ba}$ that connect them. As we mentioned, by  virtue of \cite{label4},
the hyperbolic heteroclinic cycle persists in the smooth Hamiltonian approximation (\ref{System6}) of the billiard, provided $h_0$ is large enough.
This means that system (\ref{System6}) has in every energy level $H=h\geq 1$ a pair of hyperbolic periodic orbits
$L_a(h)$ and $L_b(h)$ and the heteroclinic orbits, $\Gamma_{ab}(h)$ and $\Gamma_{ba}(h)$
where $\Gamma_{ab}$ is the transverse intesection of the unstable manifold of $L_a$ with the stable manifold of $L_b$, and similarly for $\Gamma_{ba}$:\\
$$\Gamma_{ab}\subseteq W^u(L_a)\cap W^s(L_b), \ \ \ \Gamma_{ba}\subseteq W^u(L_b)\cap W^s(L_a),$$
and these four orbits are close (at $h_0$ large enough) to the corresponding orbits of the billiard in the same energy level.

By \cite{Shilnikov}, the set $\Lambda(h)$ of all orbits which stay in a small neighbourhood (in the level set $H=h$)
of $L_a(h)\cup L_b(h)\cup \Gamma_{ab}(h) \cup \Gamma_{ba}(h)$ is a uniformly-hyperbolic set which is in one-to-one correspondence with the
set of all sequences of $a$'s and $b$'s. More precisely, in each energy level we take two small smooth cross-sections, $\Sigma_a$ and $\Sigma_b$,
to $L_a$ and $L_b$ respectively ( we assume the cross-sections depend on $h$ smoothly). Every orbit from $\Lambda$ must intersect
$\Sigma_a \cup \Sigma_b$ infinitely many times. The sequence of $a$'s and $b$'s denoting the corresponding cross-sections that the trajectory passes
through gives the code, $\{\xi_i\}_{i=-\infty}^{+\infty}$ ($\xi_i\in\{a,b\}$), of the trajectory. Namely, if $M_i$ is the sequence of points at which
an orbit from $\Lambda$ intersects $\Sigma_a \cup \Sigma_b$, then $M_i\in \Sigma_{\xi_i}$.

The flow induced by (\ref{System6}) on any given energy level $H=h$ defines Poicar\'e maps on these cross-sections, denoted by
\begin{equation}
\Pi_{cc'}:\Sigma_c\rightarrow\Sigma_{c'}, \ \ \ c\in\{a,b\}. \label{PoincareMaps}
\end{equation}
These maps are such that for a trajectory of (\ref{System6}) that intersects $\Sigma_a \cup \Sigma_b$ at a point $M_i\in\Sigma_c$ and then
at a point $M_{i+1}\in\Sigma_{c'}$ we have
\begin{equation}
M_{i+1}=\Pi_{cc'}M_i. \label{PoincareMaps2}
\end{equation}
In other words, $\Pi_{aa}$ and $\Pi_{bb}$ are the Poincare maps near the periodic orbits $L_a$ and, respectively, $L_b$, while
$\Pi_{ab}$ and $\Pi_{ba}$ correspond to a passage near the heteroclinics $\Gamma_{ab}$ and $\Gamma_{ba}$.
The theory of billiard-like potentials built in \cite{label4} implies that since the billiard orbits $L_a$,
$L_b$, $\Gamma_{ab}$, $\Gamma_{ba}$ are non-simgular, the Poincare maps $\Pi_{cc'}$ for system (\ref{System6}) are $C^r$-close, at $h_0$ large enough,
to the corresponding Poincare maps defined by the billiard flow. In particular, the hyperbolicity of these maps for the billiard flow (which follows
from the fact that the periodic orbits $L_a$ and $L_b$ are hyperbolic and the heteroclinic orbits $\Gamma_{ab}$ and $\Gamma_{ba}$ are transverse)
is inherited in system (\ref{System6}) for all $h_0$ large enough.

In fact, since each periodic orbit is saddle, one can represent the two-dimensional cross-sections $\Sigma_{a,b}$ as the cross product of two ceratin small intervals
\[\Sigma_a=U_a\times W_a,\ \ \ \Sigma_b=U_b\times W_b,\]
where $U_{a,b}\in\mathbb{R}$ correspond to contracting directions and $W_{a,b}\in\mathbb{R}$ correspond to expanding directions. By virtue of \cite{AShi},
since the heteroclinic orbits $\Gamma_{ab}$ and $\Gamma_{ba}$ are transverse, one can write the Poincar\'e maps (\ref{PoincareMaps}) in the so-called
cross form (see \cite{label8,label9}). Namely, there exist smooth functions
\begin{equation}
f_{cc'}:U_c\times W_{c'}\rightarrow U_c', \ \ \ g_{cc'}:U_c\times W_{c'}\rightarrow W_c, \label{CrossForm}
\end{equation}
(where $c,c'\in\{a,b\}$) such that a point $M_i=(u_i,w_i)\in\Sigma_{c}$ is mapped to $M_{i+1}=(u_{i+1},w_{i+1})\in\Sigma_{c'}$ by the map $\Pi_{cc'}$, if
and only if
\begin{equation}
f_{cc'}(u_i,w_{i+1})=u_{i+1},\ \ \ g_{cc'}(u_i,w_{i+1})=w_i. \label{CrossForm2}
\end{equation}
Moreover there exists $\lambda>0$ such that
\begin{equation}
\left\|\frac{\partial(f_{cc'},g_{cc'})}{\partial(u,w)} \right\|\leq\lambda<1. \label{CrossFormIneq}
\end{equation}
The latter inequality means, essentially, that the Poincar\'e maps $\Pi_{cc'}$ are contracting in the $u$-coordinate and expanding in the $v$-coordinate.

From \cite{label1} it follows that the operator
\begin{equation}
\{(u_i,w_i)\}_{i=-\infty}^{+\infty}\rightarrow\{(f_{\xi_{i-1}\xi_i}(u_{i-1},w_i),g_{\xi_i\xi_{i+1}}(u_i,w_{i+1}))\}_{i=-\infty}^{+\infty},
\end{equation}
is a contraction mapping for any code sequence $\xi=\{\xi_i\}_{i=-\infty}^{+\infty}$, which gives us
the existence and uniqueness of the orbit $L_{\xi}(h)\in\Lambda(h)$ with the given code $\xi$ (see \cite{Shilnikov,AShi}).
By fixing a code $\xi$, the family of the orbits with this code parameterised by $h$, i.e. the manifold
$$\mathcal{L}_\xi^0=\bigcup_{1\leq h\leq h_1/h_0} L_{\xi}(h)$$
for any given $h_1>h_1$, is a normally-hyperbolic invariant manifold of system (\ref{System6}). System (\ref{System5}) at $\epsilon=0$ is obtained from (\ref{System6})
simply by adding an equation $\dot\theta=\omega/\sqrt{h_0}$ for the phase $\theta\in\mathbb{S}^1$. As this equation is decoupled from
the first two, and the evolution of $\theta$ is non-hyperbolic (it is just a linear rotation),
the manifold $\mathcal{L}_{\xi}^0\times \mathbb{S}^1$ is the invariant normally-hyperbolic manifold for system (\ref{System5}) at $\epsilon=0$.
The normally-hyperbolic invariant manifolds are known to persist at small perturbations \cite{Fenichel}, which implies that (\ref{System5})
has an invariant manifold $\mathcal{L}_{\xi}(\epsilon)$ close to $\mathcal{L}_{\xi}\times \mathbb{S}^1$ for all small $\epsilon$ (one might need to take
smaller $\epsilon$ to make $h_1$ larger). Because the normal hyperbolicity is uniform, these manifolds exist for all codes $\xi$ for the same range of $\epsilon$
values. A formal proof of the existence of the invariant manifolds $\mathcal{L}_{\xi}(\epsilon)$ is achieved as follows. Note that the Poincare maps $\Pi_{cc'}$
are defined for all $\epsilon$ sufficiently small and depend smoothly on $\epsilon$, so they can be written in the form
\begin{equation}\label{pnmp}\begin{array}{l}\displaystyle
u_{i+1}=f_{cc'}(u_i,w_{i+1})+O(\epsilon),\ \ w_i=g_{cc'}(u_i,w_{i+1})+O(\epsilon),\\ \\ \displaystyle
h_{i+1}= h_i +O(\epsilon), \;\; \theta_{i+1}=\theta_i+O(\epsilon+\frac{1}{\sqrt{h_0}}),
\end{array}
\end{equation}
where the energy $h$ is defined by (\ref{hamnd}); the $O(\cdot)$-terms are functions of $u_i,w_{i+1},h_i,\theta_i$.
Now, lemma 1 of \cite{label1} is applied to these maps, which immediately
gives the existence of the invariant manifolds $\mathcal{L}_{\xi}(\epsilon)$ close to $\mathcal{L}_{\xi}\times \mathbb{S}^1$ provided $\epsilon$ is small and
$h_0$ is large enough.

By differentiating (\ref{hamnd}), we find that the rate of change of $h=H(y,z,p_y,p_z,\theta)$ along an orbit of (\ref{System5}) is given by
\begin{equation}
\frac{d h}{dt} = \epsilon(\omega A_{\omega} k(y,z)\cos\theta + \mathcal{O}(\epsilon)). \label{EnergyChange}
\end{equation}
Since the Poincare map (\ref{pnmp}) at $c=c'$ corresponds to one round near the billiard periodic orbit $L_c$, it follows from (\ref{EnergyChange}) that
given an arbitrarily small $\delta$ one can choose $\epsilon$ and $h_0^{-1}$ small enough such that
\begin{equation}
\frac{h_{i+1}-h_i}{\theta_{i+1}-\theta_i}>\epsilon (A_\omega v_c \cos\theta - \delta) \label{vcCalculation}
\end{equation}
in the map $\Pi_{cc}$ ($c=a$ or $b$). In other words, as long as an orbit of system (\ref{System5}) stays near $L_c$, the change in $h$ can be estimated by the inequality
\begin{equation}
\frac{dh}{d\theta}>\epsilon (A_\omega v_c \cos\theta - \delta). \label{vcdd}
\end{equation}
Recall, that the constants $v_a$ and $v_b$ are given by (\ref{spdvl}). By (\ref{vabc}) we may fix the choice of $L_a$ and $L_b$
such that
\begin{equation}\label{vab0}
v_a>v_b.
\end{equation}

Formally, system (\ref{System5}) is not of the form studied in \cite{label1}. However, the study in \cite{label1} (see Theorem 1 there)
is reduced to the study of Poincare maps $\Pi_{cc'}$ of the form which includes (\ref{pnmp}) as a partial case. Therefore, we may apply the results
of \cite{label1} to the study of the behaviour of the orbits of system (\ref{System5}) which belong to the invariant manifolds $\mathcal{L}_{\xi}(\epsilon)$
(these orbits intersect the cross-sections $\Sigma_a\cup \Sigma_b$, so they are defined by the Poincare maps $\Pi_{cc'}$). Namely, we have two main conclusions.
First, we consider the codes $\xi$ such that the symbols $a$ and $b$ go always in blocks of some fixed size, large enough for the transitions from a neighbourhood
of $L_a$ to a neighbourhood of $L_b$ or from a neighbourhood of $L_b$ to a neighbourhood of $L_a$ to happen relatively rarely, so the contribution of these transitions
to the change in $h$ can be neglected (more precisely, it is absorbed in the small $\delta$ term in (\ref{vcdd0})). Then, by (\ref{vcdd}),
we obtain that for the orbits that stay on the invariant manifolds $\mathcal{L}_{\xi}(\epsilon)$ with such codes
$\xi$ the evolution of $h$ is estimated by
\begin{equation}
\frac{dh}{d\theta}>\epsilon (A_\omega v_{c(t)} \cos\theta - \delta) \label{vcdd0}
\end{equation}
where $c(t)=a$ or $b$ indicates where the orbit finds itself at the moment $t$, near $L_a$ or near $L_b$. Second, by repeating the corresponding construction
in \cite{label1}, for each small $\epsilon$ we can find a code $\xi$ and an initial condition on the corresponding manifold $\mathcal{L}_{\xi}(\epsilon)$ such that
for the corresponding orbit
$$c(t)=a \quad\mbox{when}\quad A_\omega \cos\theta>0, \qquad\mbox{and}\qquad
c(t)=b \quad\mbox{when}\quad A_\omega \cos\theta<0.$$
For such orbit, we can rewrite (\ref{vcdd0}) as
$$\frac{dh}{d\theta}>\epsilon (A_\omega \cos\theta \frac{v_a+v_b}{2}+|A_\omega \cos\theta|\frac{v_a-v_b}{2}- \delta).$$
By integrating over one period of $\theta=\omega t$, we find that for this particular orbit the change of the energy $h$ over
each consecutive $2\pi/\omega$ interval of time is estimated by
$$\Delta h>\epsilon (|A_\omega| (v_a-v_b)-2\pi\delta)>0$$
(recall that $\delta$ can be taken as small as we need and $v_a>v_b$ by assumption of the theorem). Thus, the energy of the chaotic oscillator
steadily grows along the chosen orbit (at a linear rate), which completes the proof of the theorem.

\section*{Acknowledgements} We are grateful to V.Gelfreich, A.Vladimirov, and V.Rom-Kedar for useful discussions.
This work was supported by the Leverhulme Trust grant RPG-279 and the Russian Ministry of Education and Science grants 14.B37.21.0862.

\end{document}